\documentclass[10pt]{amsart}
\usepackage{epsfig,cancel,amsmath,amsthm,amssymb, todonotes,mathrsfs,ulem}%, mathabx}
\usepackage{color}

\usepackage{lscape}
\usepackage{wrapfig}
\usepackage{color}
\usepackage{graphicx,framed}
\usepackage[colorlinks=true, pdfstartview=FitV, linkcolor=blue, citecolor=blue, urlcolor=blue]{hyperref}
\definecolor{shadecolor}{rgb}{0.95, 0.95, 0.86}

\def\red#1{\textcolor[rgb]{0.9, 0, 0}{#1} }

\def\blue#1{\textcolor[rgb]{0,0,1}{#1}}

\def \QED {$\blacksquare$ \hfill \par}

\def\remove#1{\red{\cancel{#1}}}

\def\replace#1#2{\red{\cancel{{\hbox{\scriptsize  {$#1$}} }}}{\blue {#2}}}

\def\ra{\rightarrow}
\newcommand{\al}{\alpha}
\def\Id{ \mathrm{Id}} 
\newcommand{\bt}{\beta} 
\newcommand{\bet}{\beta}

\newcommand{\G}{\Gamma} 
 
\renewcommand{\k}{\varkappa} 
\renewcommand{\d}{\delta}
\newcommand{\D}{\Delta} 
 
\newcommand{\e}{\epsilon}
  
\newcommand{\g}{\gamma} 
\newcommand{\la}{\lambda}

\newcommand{\pa}{\partial}

\newcommand{\br}{{\mathbb R}}

\newcommand{\CH}{{\mathcal H}}
 
\newcommand{\CF}{\mathcal{F}}
\newcommand{\CM}{\mathcal{M}}
\newcommand{\vvvert}{|||} 
\newcommand{\deq}{=}

\newcommand{\bto}{\beta_{\text{od}}}
\newcommand{\bte}{\beta_{\text{ev}}}

\newcommand{\bff}{{\bf f}}

\def\Id{ \mathrm{Id}} 

\newtheorem{theorem}{Theorem}[section]
\newtheorem{example}[theorem]{Example}
\newtheorem{exercise}[theorem]{Exercise}

\newtheorem{lemma}[theorem]{Lemma}
\newtheorem{remark}[theorem]{Remark}
\newtheorem{problem}[theorem]{Riemann-Hilbert Problem}

\newtheorem{proposition}[theorem]{Proposition} 
\newtheorem{corollary}[theorem]{Corollary} 
\newtheorem{definition}[theorem]{Definition}
\def\ri{\right}

\def\bth{\begin{theorem}}
\def\et{\end{theorem}}
\def\bc{\begin{corollary}}
\def\ec{\end{corollary}}
\def\bx{\begin{example}\small}
\def\ex{\end{example}}
\def\bxr{\begin{exercise}\small}
\def\exr{\end{exercise}}
\def\bl{\begin{lemma}}
\def\el{\end{lemma}}
\def\bd{\begin{definition}}
\def\ed{\end{definition}}
\def\bp{\begin{proposition}}
\def\ep{\end{proposition}}

\def\br{\begin{remark}}
\def\er{\end{remark}}

\def\be{\begin{equation}}
\def\ee{\end{equation}}

\newcommand{\bs}{\begin{split}}
\newcommand{\esp}{\end{split}}
\def\ov {\overline}

\def\&{\hspace{-15pt}&}
\def\bea{\begin{eqnarray}}
\def\eea{\end{eqnarray}}
\def\beas{\begin{eqnarray*}}
\def\eeas{\end{eqnarray*}}
\def\bi{\begin{itemize}}
\def\ei{\end{itemize}}

\def \pa{\partial}
\def\C{{\mathbb C}}

\def\R{{\mathbb R}}
\def\N{{\mathbb N}}
\def\wh{\widehat}

\def\a{\alpha}

\def\l{\lambda}
\def\m{\mu}

\def\1{{\bf 1}}

\def\s{ {\sigma}}

\def\th{ {\theta}} 
\def\Th{ {\Theta}} 

\def\z{\zeta}

\def\hf{\frac{1}{2}}
\def\le{\left}

\def\Col{{\rm Col}}
\def\diag{{\rm diag}}

\newcommand{\Rscr}{\mathcal R}
\numberwithin{equation}{section}

\begin{document}
\baselineskip 18pt plus 1pt minus 1pt
%\begin{titlepage}

%\begin{center}
%\begin{large}

\title[Inversion formula and the range of a vector  multi-interval finite Hilbert transform]{Inversion formula and range conditions for a vector 
%\remove{linear system with} 
multi-interval
finite Hilbert transform in $L^2$}
\author[A. Katsevich, M. Bertola, A. Tovbis]{Alexander Katsevich$^1$, Marco Bertola$^{2,3}$ and  Alexander Tovbis$^1$}
\thanks{$^1$ Department of Mathematics, University of Central Florida, P.O. Box 161364, 4000 Central Florida
Blvd, Orlando, FL 32816-1364, USA.   }
\thanks{$^{2}$ {\it  Department of Mathematics and
Statistics, Concordia University\\ 1455 de Maisonneuve W., Montr\'eal, Qu\'ebec,
Canada H3G 1M8}}
\thanks{$^{3}$  {\it SISSA, International School for Advanced Studies, via Bonomea 265, Trieste, Italy }
}
\maketitle

\begin{abstract}
Given $n$ disjoint intervals $I_j$,  on $\R$ together with $n$ functions $\psi_j\in L^2(I_j)$, $j=1,\dots n$, and an $n\times n$ 
matrix $\Th$,
the problem is  to find an $L^2$ solution  $\vec \varphi= \Col (\varphi_1,\dots,  \varphi_n)$, $\varphi_j \in L^2(I_j)$
to the linear system $\chi\Th \CH \vec \varphi \deq \vec\psi$, where 
$\CH= {\rm diag} (\CH_1 ,\dots, \CH_n)$ is a matrix of finite Hilbert transforms and    
$\chi=\text{diag}(\chi_1,\dots,\chi_n)$ is a matrix of the corresponding characteristic functions on $I_j$, and 
$\vec \psi=\Col (\psi_1,\dots,\psi_n)$.
Since we can interpret $\chi\Th \CH \vec \varphi$ as a generalized vector multi-interval finite Hilbert
transform, we call the formula for the solution as ``the inversion formula'' and  the necessary and sufficient conditions 
for the existence of a solution as the  ``range conditions''. 
In this paper we derive the explicit inversion formula and the  range conditions
in two specific cases: a) the matrix $\Th$ is symmetric and positive definite, and; b) all the entries of $\Th$ are equal to one.
We also prove the uniqueness of solution, that is, that our
transform is injective. 
 When
the matrix $\Th$ is positive definite, the inversion formula is given in terms of the solution of the associated matrix
Riemann-Hilbert Problem. 
%and the range conditions smoothly depend on the  endpoints of the intervals. 
We also discuss other cases of the matrix $\Th$.
\end{abstract}

\baselineskip 12pt

\section{Introduction}\label{sec-intro}

We start by reminding  the reader of  the  well known inversion formula and range condition for a finite Hilbert transform $\CH$ in $L^2$, see, 
for example, \cite{OkE}. 
Let $I=[\a,\bt]\subset \R$. The finite Hilbert transform   $\CH:~L^2(I)\mapsto L^2(I)$ is defined by 
\be\label{Hilb}
\CH [f](z): =\frac1{\pi } \int_I \frac{ f(\z)}{\z-z} d\z.
\ee
Here and everywhere below, whenever $z$ belongs to the interval of integration, the integral is understood in the Cauchy Principal Value sense. In all other cases, this is just an ordinary integral.
The following  facts about $\CH$ are known:
\begin{enumerate}
 \item The operator $\CH: L^2(I)\mapsto L^2(I)$ is injective  and its range is a proper dense subspace of  $ L^2(I)$
 so that $\CH$ is not a Fredholm operator;
 \item Let $f\in  L^2(I)$. Then $f$ is in the range 
 %$\Rscr\subset L^2(I)$ 
 of $\CH$ if and only if there exists a unique constant $\varkappa\in\C$
 such that
 \be\label{R-1}
-\frac1{\pi R(z)} \int_I \frac{R(\z) f(\z)}{\z-z} d\z- \frac\varkappa{R(z)}     \in  L^2(I),
\ee
 where $R(z)=\sqrt{(z-\bt)(z-\a)}$ and $R(z)\sim z$ as $z\ra\infty$;
 
 \item If $f$  is is in the range  of $\CH$ 
 then
\be\label{inv_H_gen}
\CH^{-1} [f](z) =-\frac1{\pi R(z)} \int_I \frac{R(\z) f(\z)}{\z-z} d\z- \frac\varkappa{R(z)}, 
\ee
where the constant $\varkappa$ is the same as in \eqref{R-1}.

\end{enumerate}

These results can be generalized in several directions: for example, one can consider  functional spaces $L^p(I)$, $p>1$, see \cite{OkE},
or more general singular integral transforms, like, for example, the  $\cosh$-transform (\cite{BKT13}). {\it In the current paper, we  extend the inversion
formula \eqref{inv_H_gen} and the range condition \eqref{R-1} to the case of the  vector multi-interval finite Hilbert transform $\CH$.}
Here is the general setting of the problem.
Given: 
\bi
\item
$n\in\N$ disjoint intervals $I_j=[\a_j,\bt_j]$ on $\R$, $-\infty<\a_1<\bt_1<\a_2<\dots\a_n<\bt_n<\infty$;

\item $n$ {real-valued} functions $\psi_j\in L^2(I_j)$, which we represent as $\vec \psi=\Col (\psi_1,\dots,\psi_n)$;

\item 
an 
%non-singular 
$n\times n$ real matrix $\Theta=(\th_{jk})$,
\ei
find the range conditions, the inversion formula and study the uniqueness of solution  for the following system of singular integral equations 
\be\label{mat-eq}
\chi\Th \CH \vec \varphi \deq \vec\psi
\ee
on $I=\cup_{j=1}^nI_j$, where $\vec \varphi= \Col (\varphi_1,\dots,  \varphi_n)$, $\varphi_j \in L^2(I_j)$, and
\be\label{def-Hj}
\CH= {\rm diag} (\CH_1 ,\dots, \CH_n),\ \chi=\text{diag}(\chi_1,\dots,\chi_n).
\ee
Here $\CH_j:~L^2(I_j)\mapsto L^2(I)$ denotes the finite Hilbert transform (FHT), which integrates on $I_j$ and evaluates on $I$. Also, $\chi_j$ denotes the characteristic function of $I_j$, $j=1,2,\dots,n$. Whenever appropriate, $\chi_j$ can also be viewed as restriction operators.
%, and $\deq$ denotes that the equality between the $j$-th rows on the left and on the right in \eqref{mat-eq} holds only on $I_j$, $j=1,\dots,n$:
%\be\label{mat-eq-equiv}
%(\Th \CH \vec \varphi)_j(z) = \psi_j(z),\ z\in I_j,\ j=1,\dots,n.
%\ee
%We are also interested to know if the range conditions ``smoothly'' depend  on the geometry of the intervals, that is, on $\a_j,\bt_j$. 
To the best of our knowledge, this is the first paper that analyzes the equation \eqref{mat-eq}.

The problem mentioned above appears naturally in some applications most notably in the study of discrete $\bt$-ensembles,
see, for example,  \cite{BGG}, where $\Th$ is considered as the ``interaction'' matrix, and also in the study of lozenge tilings of polygonal regions.  Solution of these problems depends 
significantly on the particular  matrix $\Th$ in \eqref{mat-eq}.
Most of the results  obtained in this paper results  apply to the case where $\Th$ is a   positive definite symmetric matrix, but we will also consider
some other situations, for example, the case of all $\th_{jk}=1$ (uniform interactions), where $\th_{jk}$ is the $j,k$th entry of $\Th$. 
The analysis of the uniform interaction matrix 
case can be found in Section \ref{sec-unif}.

The main case  of a positive definite symmetric matrix $\Th$ in \eqref{mat-eq} is considered in Sections \ref{sec-2a}-\ref{sec-2}.
The main steps  of our  approach are as follows: we first reduce the operator $\chi \Th \CH $ from \eqref{mat-eq}
to $\Id -\wh K$, where $\wh K$ is a Hilbert-Schmidt operator, and then use the approach of \cite{IIKS} to find the resolvent
of $\Id -\frac{\wh K} \l$ in terms of the solution $\G(z)$ of the corresponding matrix Riemann-Hilbert Problem (RHP). This problem is determined  by
the intervals $I$ (geometry) and by the matrix $\Th$.
Of course, along the way,
we have to prove the invertibility of $\Id -\wh K$.

The plan of the paper is the following. In Section \ref{sec-2a} we obtain the $L^2$ solution of \eqref{mat-eq} (the  inversion formula)
provided that such a solution exists, see Theorem \ref{theo-inv-pos-def}. This solution is expressed in terms of the matrix $\G(z)$ and the 
inverse Hilbert transform $\vec \nu(z)$ of the ``modified'' right hand side $\vec \psi(z)$ of \eqref{mat-eq}. We also state  necessary 
conditions for solvability of  \eqref{mat-eq}, see Lemmas \ref{lem-N1}, \ref{lem-N2}. The condition in Lemma \ref{lem-N1} is expressed in terms of the right-hand side of \eqref{mat-eq},
but the condition in Lemma \ref{lem-N2} is expressed in terms of the solution $\vec \varphi$ of \eqref{mat-eq}.
In Section \ref{sec-3} we prove that conditions of Lemmas \ref{lem-N1}, \ref{lem-N2} are also sufficient,  see Theorem \ref{theo-main}, i.e., they
form the {\it range condtions} for \eqref{mat-eq}. We also express the condition in Lemma \ref{lem-N2}  in terms of $\G(z)$ and $\vec \nu(z)$, see 
\eqref{N2}
and discuss its dependence
%It follows from \eqref{N2} that    the range conditions
%\replace
%{smoothly}
%{analytically} 
 on the geometry of the intervals, that is, on the points $\a_j,\bt_j$. The invertibility of the operator $Id -\wh K$ and the uniqueness of the solution of \eqref{mat-eq} are proven in Section \ref{sec-2}. 

{In Section \ref{sec-other} we gradually relax the  
requirements of  positive definiteness and  symmetry of $\Th$. Replacing   positive definiteness with a much weaker requirement that 
$\Th_d=\text{diag}\Th$ is invertible, we keep all the above mentioned main results provided that the operator  $\Id -\wh K$ is 
invertible. Next, removing the symmetry requirement of $\Th$, we end up with a more complicated expression for the second range condition in Theorem
\ref{theo-main}. Finally, we briefly discuss relaxing the requirement that $\Th_d$ is invertible, which leads to certain  analyticity requirements
on some components of $\vec\psi$. Moreover, some components $\varphi_j$ can be found as jumps over the corresponding intervals $I_j$
of the analytic continuations of the corresponding $\psi_k$ and their finite Hilbert transforms.}

The inversion formula and the range condition for \eqref{mat-eq} with the uniform interaction matrix {(i.e., $\theta_{jk}=1$ for all $j,k$)} can be found in  Section \ref{sec-unif}, see Theorem \ref{theo-unif}. The  method in this section involves diagonalization of a multi interval FHT, which is based on a special change of variables followed by the application of the Fourier transform. This result is closely related to the classical work done in the 50-s and 60-s, most notably \cite{rosenblum66}. See also  \cite{koppinc59, kopp60, widom60, kopp64, pincus64, putnam65} for related earlier work. In these papers the authors use complex-analytic methods to diagonalize and study spectral properties of certain singular operators related to the Hilbert transform. Our approach, while less general, allows to obtain a simple inversion formula and the range condition for the multi interval FHT very quickly and using completely elementary means. These two results, the inversion formula and the range condition, appear to be new.
 
The authors are grateful to Percy Deift for mentioning the problem and useful discussions.

%\blue{Some ideas for the intro. Motivation: 1) generalization of the range conditions from \cite{OkE} from one equation
%to a system; 2) works of V. Gorin et al.  Cases of $\Th$ to consider: 1) all entries of $\Th$  are  $ 1$;  b) the  case
%of $\Th=\1+\hf\s_1$ (Marco); 3) $\Th$ is symmetrical and positive definite; 4) if $n=2$ we can drop symmetry; 5) allow $\Th$ 
%to be positive semi-definite with certain other restrictions (Sasha).}

\section{Solution to \eqref{mat-eq} for positive definite symmetric $\Th$}\label{sec-2a}
%\todo{M:By definition a positive definite matrix is symmetric}
In this section, we will use the following  result of \cite{OkE} for a finite interval $I=[\a,\bt]$.
%\begin{enumerate}
%\item 
\bp\label{prop-R3}
If $f\in  L^2(I)$  and  $\frac{f(z)}{R(z)}\in L^1(I)$ then the range condition \eqref{R-1} can be written as follows
\be\label{R-3}
\int_I \frac{f(\z)d\z} {R(\z)}= 0.
%~~~~~\text{ provided  the integral does exist.} 
\ee
If $f$ is in the range, the  inversion formula  \eqref{inv_H_gen} can be replaced by
\be\label{inv_H}
\CH^{-1} [f](z) =-\frac{R(z)}{\pi } \int_I \frac{ f(\z)}{R(\z)(\z-z)} d\z.
\ee
\ep
Using \eqref{inv_H}, we immediately obtain that
\be\label{H-inv[1]}
\CH^{-1} [1](z)=0, ~~~z\in I,~~~~~{\rm and}~~~~\CH^{-1} [1](z)=-i, ~~~~~~z\in \C\setminus I,
\ee
%\end{enumerate}
where by $\CH^{-1} [1](z)$, $z\not\in I$, we understand the expression \eqref{inv_H} evaluated at $z$.

Let us split 
\be\label{Q-dec}
\Th =\Th_d+\Th_o,
\ee
where 
\be\label{Th_od}
\Th_d:=\text{diag}\Th,\quad \Th_o:=\Th-\Th_d.
\ee
%{\it In this and the following sections we assume that $\Th_d$ is invertible.}
%\red{\sout{
For the time being we assume that  $\Th_d$ is invertible.
%\blue{Note that since $\Theta$ is positive definite, the diagonal part  $ \Theta_d$ has only strictly positive entries and is thus invertible.}

Assume that a solution $\vec\varphi\in L^2(I)$ to \eqref{mat-eq} exists. Then, according to Theorem \ref{theo-inv}, proven below, the solution  $\vec\varphi$ is unique and equation  \eqref{mat-eq} can be written as 
\be\label{mat-eq-sep}
\chi\Th_d \CH \vec \varphi + \chi\Th_o \CH \vec \varphi \deq \vec\psi.
\ee

\br\label{rem-funct-not}
A function $\varphi(z)\in L^2(I)$ can be uniquely identified with a vector-function 
$\vec\varphi=\Col(\varphi_1,\dots,  \varphi_n)$, where $\varphi_j(z)=\chi_j(z)\varphi(z)$. With a mild abuse of notations, we will consider $\varphi(z)$ and $\vec\varphi $ as the same object in this paper, and use these notations intermittently, as is convenient in a given context. %\todo{M: this remark is a repeat of the sentence after (2.6). Move it there?}
\er

Let $\Rscr \subset L^2(I)$
%:=\oplus_{j=1}^n \Rscr_j \subset \oplus_{j=1}^n L^2(I_j)$ 
denote the 
range of the operator $\chi\CH$. 
%Here, with a slight abuse of notation, we assume that a vector whose $m$-th component is a function defined on $I_m$, can be viewed as a single function defined on $I$ \blue{whose restriction to $I_m$ yields the $m$-th component.} 
Note that  $\vec\psi\in\Rscr$ if and only if $\chi\Th_o \CH \vec \varphi\in\Rscr$  given that the first term in \eqref{mat-eq-sep} is automatically in the range. However,
since the $m$-th row of 
$\Th_o \CH \vec \varphi$
%, see \eqref{Th_od} for the definition of $\Th_o$, 
is analytic on
$I_m$,  
%$j=1,\dots,n$. 
  there exists, according to the range condition \eqref{R-3}, a constant vector
$\vec c\in\R^n$, such that 
\be\label{off-d-ran}
\chi\Th_o \CH \vec \varphi-\vec c\in \Rscr.
\ee
Therefore, there must be a  constant vector
$\vec c\in\R^n$, such that  $\vec\psi -\vec c\in\Rscr$.
%where $\Rscr \subset L^2(I)$
%:=\oplus_{j=1}^n \Rscr_j \subset \oplus_{j=1}^n L^2(I_j)$ 
%denotes the range of $\chi\CH$.} 
%\textcolor{red}{Incorrect. $\CH$ contains non-diagonal terms.} 
%{Here $\Rscr_j$ is the range of the FHT $\CH_j:\, L^2(I_j)\to L^2(I_j)$}
%\textcolor{red}{Need to provide different explanation in order not to conflict with the previous definition of $\CH_j$.}
%Subtracting $\vec c$ on both sides of \eqref{mat-eq}, we obtain that $\vec\psi -\vec c\in\Rscr$. 
The uniqueness of such $\vec c$ follows from the fact that any nonzero constant vector $\vec c$ cannot belong to the range, $\vec c\not\in\Rscr$. Thus, we have proved the following statement.

\bl\label{lem-N1}
 If equation \eqref{mat-eq} is solvable, then there exists a unique constant vector
$\vec c\in\R^n$ such that $\vec\psi -\vec c\in\Rscr$.
\el

Given $\vec\psi\in L^2(I)$, we denote by $\vec c[\vec\psi]$ the vector $\vec c$ from 
Lemma \ref{lem-N1}. { The existence of $\vec c[\vec\psi]$ is the first necessary condition for
the solvability of \eqref{mat-eq} in $L^2(I)$}. The  second necessary condition, given 
in Lemma \ref{lem-N2} below, follows from the above mentioned arguments.

\bl \label{lem-N2}
If $\vec\varphi\in L^2(I)$ is the solution of \eqref{mat-eq} then
\be\label{range-2}
\chi\Th_o \CH \vec \varphi-\vec c[\vec\psi]\in\Rscr.
\ee
\el

%The proof of this lemma is obvious. 
Rewriting equation \eqref{mat-eq} as follows
\be\label{mat-eq1}
 \chi\Th_d \CH \vec \varphi+\left(\chi\Th_o \CH \vec \varphi-\vec c[\vec\psi]\right) \deq \vec\psi-\vec c[\vec\psi],
\ee
we obtain 
\be\label{mat-eq2}
 \vec \varphi+ \CH^{-1}\left(\Th^{-1}_d \Th_o \CH \vec \varphi-\Th^{-1}_d\vec c[\vec\psi]\right) \deq \vec\nu,\quad 
\vec\nu:= \CH^{-1}\Th^{-1}_d \left(\vec\psi-\vec c[\vec\psi]\right),
\ee
where $\CH^{-1}:={\rm diag}(\CH^{-1}_1,\dots, \CH^{-1}_n)$.
Here $\CH^{-1}_j$ is given by \eqref{inv_H_gen} 
with $I=I_j$.
%The invertibility of $\Th_d$ follows from positive-definiteness of $\Th$. 
%Let $\chi_j$ denote the characteristic function of $I_j$, $j=1,2,\dots,n$. Denote also $\chi=\text{diag}(\chi_1,\dots,\chi_n)$.

\br \label{rem-notat}
Note that in view of our definition of $\CH_j$, here and everywhere below, the operator $\CH^{-1}_j$ is the inverse not to
$\CH_j$, but to $\chi_j\CH_j$, and the latter coincides with the usual finite Hilbert transform on $I_j$. Note also that whenever $\CH^{-1}_j$ acts 
on a function defined on all of $I$, it is applied only to its restriction on $I_j$.
\er

According to the first formula in \eqref{H-inv[1]}, equation
\eqref{mat-eq2} becomes
\be\label{mat-eq3}
 \vec \varphi+ \CH^{-1} \Th^{-1}_d \Th_o \CH \vec \varphi \deq \vec\nu,
\ee
or, in operator form,
\be\label{operKn}
(\Id - \wh K )\vec\varphi = \vec\nu, \qquad \text{where} \quad \wh K= 
\le[\begin{array}{cccc}
 0 & -\frac{\th_{12}}{\th_{11}}\CH^{-1}_1\CH_2 &\dots& -\frac{\th_{1n}}{\th_{11}}\CH^{-1}_1\CH_n\\
 -\frac{\th_{21}}{\th_{22}}\CH^{-1}_2\CH_1 & 0&\dots&  -\frac{\th_{2n}}{\th_{22}}\CH^{-1}_2\CH_n\\
 \dots&\dots&\dots&\dots \\
 -\frac{\th_{n1}}{\th_{nn}}\CH^{-1}_n\CH_1 & -\frac{\th_{n2}}{\th_{nn}}\CH^{-1}_n\CH_2&\dots&0\\
\end{array}\ri].
\ee

Let $R_j(z)=\sqrt{(z-\a_j)(z-\bt_j)}$, where $R_j(z)\sim z$ as $z\ra\infty$.  We need the following statements.

\bl \label{lem-H-1H}
For any $\phi\in L^2(I_k)$ and $j\not=k$ we have 
\be \label{H-1H}
-\CH^{-1}_j\CH_k\phi =\frac{R_j(z)}{\pi i} \int_{I_k}\frac{\phi(x)dx}{(x-z)R_j(x)},
\ee
where $z\in I_j$, and $\CH_j^{-1}$ can be represented in the alternative form \eqref{inv_H} (with $I$ replaced by $I_k$).
\el

\begin{proof}
 
According to \eqref{inv_H},
\be\label{comp-ker-or}
-\CH^{-1}_j\CH_k\phi= \frac{R_j(z)}{\pi^2}\int_{I_j}\frac{dy}{(y-z)R_{j+}(y)}\int_{I_k}\frac{\phi(x)dx}{(x-y)}
=\frac{R_j(z)}{\pi^2}\int_{I_k}\phi(x)dx   \int_{I_j}\frac{dy}{(x-y)(y-z)R_{j+}(y)}.    
\ee

 Using the identity $\frac{1}{(x-y)(y-z)}=\frac{1}{x-z}\le[ \frac{1}{x-y}+\frac{1}{y-z}\ri]$ and \eqref{H-inv[1]}
 %\be
 %\int_{I_j}\frac{dy}{(y-z)R_{j+}(y)}=0 \quad \text{when}\, z\in I_j, \qquad  \int_{I_j}\frac{dy}{(x-y)R_{j+}(y)}=-\frac{i\pi}{R_j(x)}
%\quad \text{when}\, x\not\in I_j,\ee
we calculate 
\be\label{comp-ker}
\CH^{-1}_j\CH_k\phi=-\frac{R_j(z)}{\pi i}\int_{I_k}\frac{\phi(x)dx}{(x-z)R_j(x)}.
\ee
 \end{proof}

According to Lemma \ref{lem-H-1H}, the operator $\wh K$ can be written as 
\be\label{operKn1}
 \wh K \vec\varphi(z)= 
%\frac{1}{i\remove{\pi}}
-i\chi
\le[\begin{array}{cccc}
 0 & \frac{\th_{12}R_1}{\th_{11}}\CH_2\le[\frac{\varphi_2}{R_1}\ri] &\dots& \frac{\th_{1n}R_1}{\th_{11}}\CH_n\le[\frac{\varphi_n}{R_1}\ri]\\
 \frac{\th_{21}R_2}{\th_{22}}\CH_1\le[\frac{\varphi_1}{R_2}\ri] & 0&\dots&  \frac{\th_{2n}R_2}{\th_{22}}\CH_n\le[\frac{\varphi_n}{R_2}\ri]\\
 \dots&\dots&\dots&\dots \\
 \frac{\th_{n1}R_n}{\th_{nn}}\CH_1\le[\frac{\varphi_1}{R_n}\ri] & \frac{\th_{n2}R_n}{\th_{nn}}\CH_2\le[\frac{\varphi_2}{R_n}\ri]&\dots&0\\
\end{array}\ri]=
\int_I K(z,x)\varphi(x)dx,
\ee
where the kernel $K$ of the integral operator $\wh K: L^2(I)\mapsto L^2(I)$ is given by
\be\label{kernKn}
K(z,x)=2\frac{\sum'_{j, k}\frac{\th_{jk}R_{j+}(z)}{\th_{jj}R_j(x)}\chi_j(z)\chi_k(x)}{2\pi i(x-z)}.
\ee
Here the prime notation in the summation  symbol in \eqref{kernKn}  means that  $j,k=1,\dots, n$, $k\neq j$. It is clear that 
$\wh K$ is a Hilbert-Schmidt operator since $\int_{I\times I} |K(z,x)|^2 dz dx< \infty$.

The following lemma requires that $\Th$ is a  positive-definiteness symmetric matrix. 
%\blue{\sout {Of course, positive-definiteness of $\Th$ implies that  $\Th_d$ is invertible.} }\todo{ Already said.}

\bl\label{lem-specK}
 If $\Theta$ is positive definite then $\l=1$ is not in the spectrum of $\wh K$. 
\el

\begin{proof} If $\l=1$ is in the spectrum of $\wh K$ then $\l=1$ must be an eigenvalue of $\wh K$, so that there is a nontrivial $f\in L^2(I)$ satisfying $\wh K f =f$.
 %We can write the corresponding equation as $\wh K \vec f=\vec f$. 
 Applying
 the operator $\Th_d\CH$ to both parts of $\wh K f =f$ , we obtain equation $\chi \Th \CH \vec f = \vec c$  with $\vec c=\vec c[\vec\psi]$
 satisfying \eqref{range-2}.
 However, according to Theorem \ref{theo-inv}, the latter equation  has only  trivial solution in $L^2(I)$.
 The argument is   completed.
\end{proof}

The solution to the equation \eqref{operKn} is given by 
\be\label{inv-0}
\vec\varphi= \le(\Id - {\wh K} \ri)^{-1}\vec\nu=  (\Id + \wh R(1) )\vec\nu,
\ee
where $\wh R(\l)$ denotes the resolvent operator for $\wh K$. The resolvent $\wh R$ is defined by the equation
\be\label{whR}
(\Id +\wh R)\le(\Id - \frac 1\l \wh K\ri)=\Id.
\ee
According to Lemma \ref{lem-specK}, $\wh R(\l)$ is analytic at $\l=1$.

In order to construct the resolvent kernel ${\bf R}(z,x,\l)$ of $\wh R$  we use the approach of \cite{IIKS}, see also \cite{BKT16}.
First observe that 
\be\label{IIKs}
K(z,x)=\frac{\vec f^t(z)\vec g(x)}{2\pi i(z-x)}~~~{\rm and}~~~\vec f^t(z)\vec g(z) \equiv 0,\  z\in I,
\ee
where
\be\label{f,g}
\begin{split}
&\vec f(z):= {-}2\text{Col}\le({R_{1+}(z)}\chi_1(z),\dots, {R_{n+}(z)}\chi_n(z) \ri),\\
&\vec g(x):=  \text{Col}\le(\sum_{k\neq 1} \frac{\th_{1k}}{\th_{11} R_1(x)} \chi_k(x), \dots,
 \sum_{k\neq n} \frac{\th_{nk}}{\th_{nn} R_n(x)}  \chi_k(x)\ri),
\end{split}
\ee
and   $\vec g^t$ denotes the transposition of $\vec g$. Integral kernels $K$ of the form \eqref{IIKs} are called {\it integrable}. Equations \eqref{f,g} in 
the vector form become
\be\label{fg-vect}
\vec f(z)={-}2\chi(z)R(z)e,\ \vec g(z)=R^{-1}(z)\Th_d^{-1}\Th_o\chi(z)e,
\ee
where
\be\label{R}
e=\text{Col}(1,1\dots,1),\quad R(z)=\text{diag}(R_1(z),\dots,R_n(z)).
\ee

Consider the following matrix Riemann-Hilbert Problem (RHP).

\begin{problem}
\label{RHPGamma}
Find an $n\times n$ matrix-function $\G=\G(z;\l)$, $\l\in\C\setminus\{0\}$, which is: a) analytic in 
${\C}\setminus I$;  b) $\G(\infty;\l)=\1$; c) admits non-tangential boundary values   from the upper/lower 
half-planes that belong to $L^2_{loc}(I)$, 
%in the interior points of
 and; d) satisfies the following jump condition on $I$
\be
\label{rhpG_mat}
%\begin{split} 
\G_+(z;\l)=\G_-(z;\l)\le(\1-\frac{1}{\l}\vec f(z)\vec {g^t}(z)\ri) 
=\G_-(z;\l)V(z;\l).
%\label{endpcond-out}
%&&\G(z;\la)=\le[\mathcal O(1), \mathcal O((z-a_j)^{-\hf})\ri],~~z\ra a_j,~~j=1,2g+2,\\
%\label{endpcond-out-inn}
%&&\G(z;\la)=\le[\mathcal O(1), \mathcal O(\ln(z-a_j))\ri],~~~~z\ra a_j,~~j=2,2g+1,\\
%\label{endpcond-inn}
%&&\G(z;\la)=\le[\mathcal O(\ln (z-a_j)),\mathcal O(1)\ri], ~~~~~ z\ra a_j,~~j=3, \dots, 2g.
\ee
%Here the endpoint  behavior of $\Gamma$ is described column-wise. 
\end{problem}
For  convenience, we will frequently omit the dependence on $\l$ in the notations.

\br\label{rem-uniG}
Using standard arguments, one can show that the solution to the RHP \ref{RHPGamma}, if it exists, is unique.
\er

In more explicit terms, we note that the jump matrix in \eqref{rhpG_mat} is given by:
\be\label{jumpG}
V(z;\l)= \le[\begin{array}{cccc}
 1 & 2\frac{\th_{12}R_{1+}(z)}{\l\th_{22}R_{2}(z)}\chi_1(z)&\dots & 2\frac{\th_{1n}R_{1+}(z)}{\l\th_{nn}R_{n}(z)}\chi_1(z) \\
2\frac{\th_{21}R_{2+}(z)}{\l\th_{11}R_{1}(z)}\chi_2(z) & 1 &\dots & 2\frac{\th_{2n}R_{2+}(z)}{\l\th_{nn}R_{n}(z)}\chi_2(z) \\
 \dots&\dots&\dots&\dots \\
2\frac{\th_{n1}R_{n+}(z)}{\l\th_{11}R_{1}(z)}\chi_n(z) &   2\frac{\th_{n2}R_{n+}(z)}{\l\th_{22}R_{2}(z)}\chi_n(z) &\dots&1\\ 
\end{array}\ri].
\ee

\br\label{rem-beh-of -col}
Let $\G_j(z)$ denote the $j$ th column of $\G(z)$ and $\D_k(\vec f)$ denote the jump of the vector-function  $\vec f=\vec f(z)$ over $I_k$,
$j,k=1, dots, n$. It follows from \eqref{rhpG_mat}, \eqref{jumpG} that 
\be\label{jump-col}
\D_j\G_j=0 \quad \text{ and} ~~ 
\D_k\G_j=2\frac{\th_{kj}R_{k+}}{\th_{jj}R_{j}}\G_{k-}, ~k\neq j.
\ee
The first equation of \eqref{jump-col} together with the requirement c) from the RHP \ref{RHPGamma} imply that $\G_j$ is analytic in
$\bar\C\setminus\cup_{m\neq j}I_m$
 for any $j=1,\dots,n$. The second equation of \eqref{jump-col} yields
 \be\label{Gj-Cauchy}
 \G_j(z)=\frac{1}{\pi i}\sum_{k\neq j}\int_{I_k}\frac{\th_{kj}R_{k+}(\z)\G_k(\z)d\z }{\th_{jj}R_{j}(\z)(\z-z)}, 
\ee
which implies that $\G_{j\pm}(z)$ are analytic in the interior of $I_k$ and are bounded at its endpoints $\a_k,\bt_k$.
Thus, we showed that $\G_\pm (z)$ are bounded at all the endpoints and analytic in the interior  of each $I_j$. 
 \er

The utility of the RHP \ref{RHPGamma} is demonstrated by the  following lemma. 

\bl\label{lem-kernR}
If $\l$ is such that the solution $\G(z;\l)$ of the RHP \ref{RHPGamma} exists, then the kernel 
${\bf R}$ of the resolvent $\wh R$ defined by \eqref{whR} 
is given by 
\be
\label{resolvent}
{\bf R}(z,x;\l) =   
\frac{
\vec {g^t}(x) \Gamma^{-1}(x;\l) \Gamma(z;\l)\vec f(z)} {2\pi  i \l  (z-x)}.
\ee
\el

The proof of this lemma can be found, for example, in  \cite{BKT16}, Lemma 3.16.

\br\label{rem-non-sing}
Note that due to the second equation in \eqref{IIKs}: a) the integral kernel ${\bf R}(z,x;\l)$ given by \eqref{resolvent} is non-singular; 
and b) {
$ \Gamma_+(z)\vec f(z)= \Gamma_-(z)\vec f(z)$ and $\vec {g^t}(z) \Gamma_+^{-1}(z) =\vec {g^t}(z) \Gamma_-^{-1}(z)$ on $I$. Thus,
it does not matter whether we use $\G_+$ or $\G_-$ in the equation \eqref{resolvent}.
%\blue{
Indeed, according to \eqref{rhpG_mat},
\bea\label{gGGf}
&&\Gamma_+^{-1}(z;\l)=\le(\1+\frac{1}{\l}\vec f(z)\vec {g^t}(z)\ri)\Gamma_-^{-1}(z;\l)~~\text{ on $I$, so}~~\cr
&&\vec {g^t}(z)\Gamma_+^{-1}(z;\l)=\vec {g^t}(z)\le(\1+\frac{1}{\l}\vec f(z)\vec {g^t}(z)\ri)\Gamma_-^{-1}(z;\l)=\vec {g^t}(z)\Gamma_-^{-1}(z;\l)
\eea
on $I$. Similarly, we can show that $\Gamma(z)\vec f(z)$ has no jump on $I$.}
\er

As we have seen, the existence of $\G(z;\l)$, $\l\in \C\setminus\{0\}$, implies that $\l$ is a regular (non-spectral) point of $\wh K$. In fact,  the converse is also true, as is shown by the following lemma.

%\section{From resolvent to RHP}

\bl\label{lem-exist-G}
$\l\in \C\setminus\{0\}$ is a regular point of the operator $\wh K$ with an integrable kernel
{if and only if} the solution $\G(z;\l)$ to the RHP \ref{RHPGamma} exists.
\el
\begin{proof}
In view of Lemma \ref{lem-kernR} it is sufficient to prove  
the existence of $\G(z;\l)$ for  any regular (non-spectral) point $\l$ of $\wh K$.
%can be found, for example, in Lemma 3.16, \cite{BKT16}. Below, we prove the converse. 
%Suppose we have  an integrable kernel
%\be 
%K(x,y) = \frac {\vec f^t(x) \cdot \vec g(y)}{x-y}.,\ \ \ \vec f^t(x) \cdot \vec g(x ) \equiv 0, \ \  \ x, y\in I
%\ee
Suppose the operator $(\Id -\frac{\wh K}{\l})$  with the kernel given by \eqref{IIKs}, \eqref{f,g} is invertible. Write the inverse as $\Id + \wh R(\l)$. Define 
\be
\vec F(z;\l)  := \le(\Id -\frac{\wh K}{\l}\ri)^{-1} \vec f(z) = (\Id +\wh R(\l)) \vec f(z)
\label{defF}
\ee
and define the matrix 
%(I am assuming the contour does not contain $\infty$, otherwise we need a different normalization point)
\be
\G(z;\l):= \1 - \int_I \frac {\vec F(w;\l) \vec g^t(w) { d} w}{2\pi i\l(w-z)}.
\label{defG}
\ee
From \eqref{defG} and the Plemelj-Sokhotski theorem, 
\be
\G_+(z;\l) - \G_-(z;\l) =  -\frac 1\l\vec  F(z;\l)\vec g^t(z),\ z\in I,
\label{JG_1}
\ee
which implies
\be\label{JG_2} 
{\G_+(z;\l) \vec f(z)} = \G_-(z;\l)\vec f(z) - \frac 1\l\vec  F(z;\l)\vec g^t(z)\vec f(z)={ \G_-(z;\l)\vec f(z)},\ z\in I.
\ee
Thus, $\G(z;\l) \vec f(z)$ has no jump across $I$. 
%\textcolor{red}{Prove similar statement of $g$ and make a remark about both $f$ and $g$.} 
We used \eqref{IIKs} in \eqref{JG_2}.
%(To simplify notation here, we drop the spectral variable $\l$.) 
On the other hand, from \eqref{IIKs}, \eqref{defF}, and the definition \eqref{defG}, we also have 
\be\begin{split}
\G_\pm(z;\l) \vec f(z) &= \vec f(z)  -  \int_I \frac {\vec F(w;\l) \overbrace{\vec g^t(w)\vec f(z) }^{\hbox{\tiny{ scalar}}} {\rm d} w}{2\pi i\l(w-z)_-}
= \vec f + \frac 1\l \wh K \vec F \\
&=\vec f - \le(\Id - \frac {\wh K} \l\ri) \vec F +  \vec F
= \vec f - \le(\Id - \frac{\wh K}\l\ri)(\Id+\wh R(\l)) \vec f +  \vec F  = \vec F(z;\l).
\end{split}
\ee
Therefore \eqref{JG_1} becomes 
\be
\G_+(z;\l) - \G_-(z;\l) =  -\frac{1}{\l}\G_-(z;\l)\vec  f(z)\vec g^t(z)\ \ \Leftrightarrow \ \ \ 
{\G_+(z;\l) = \G_-(z;\l)\le( \1 - \frac{1}{\l}\vec  f(z)\vec g^t(z)\ri)},
\ee
so we recover the jump condition \eqref{rhpG_mat} for $\G(z;\l)$. The remaining requirements of the RHP \ref{RHPGamma} follow
from the definitions \eqref{defF} and \eqref{defG}.
\end{proof}

%olution  Although we do immediately use the following statement, the following lemma  

We are now ready to formulate the inversion formula.

\begin{theorem}\label{theo-inv-pos-def}
 If $\Th$ is a positive definite symmetric matrix and if the solution $\varphi\in L^2(I)$ to \eqref{mat-eq} exists, then
\be\label{sol-phi}
\varphi(z)=\nu(z)+\int_I \frac{
\vec {g^t}(x) \Gamma^{-1}(x) \Gamma(z)\vec f(z)\nu(x)dx} {2\pi  i   (z-x)},
\ee
 where $\G(z)=\G(z;1)$ solves the RHP \ref{RHPGamma}, and $\nu$  and $\vec f, \vec g$   are defined by \eqref{mat-eq2} and \eqref{f,g}, respectively.
\end{theorem}

\begin{proof}
According to Lemma \ref{lem-specK}, $\l=1$ is a regular point of the operator $\wh K$. Therefore, according to Lemmas \ref{lem-kernR} and \ref{lem-exist-G}, $\G(z)$ exists, and the resolvent kernel is given by \eqref{resolvent}. Thus, the solution of \eqref{mat-eq3} is given by \eqref{sol-phi}.
\end{proof}
  
\br The equation \eqref{sol-phi} can be written component-wise as 
\be\label{sol-phi-comp}
\varphi_m(z)=\nu_m(z)+\frac{R_m(z)}{\pi i}\sum_{k\neq m}
\int_{I_k} \sum_{j\neq k}\frac{\th_{jk}{\Gamma^{-1}_{kj}(x)\Gamma_{jm}(z)}\nu_j(x)dx}{\th_{jj}R_j(x)(x-z)},\ z\in I_m,\,m=1,2,\dots,n.
\ee
\er
%\red{[Is that true?]}

%According to Theorem \ref{theo-kernR}, solution to \eqref{operKn} is given by
%\be\label{sol-phi}
%\varphi(z)=\nu(z)+\int_I \frac{
%\vec {g^t}(x) \Gamma^{-1}(x) \Gamma(z)\vec f(z)\nu(x)dx} {2\pi  i   (z-x)}.
%\ee
%Substituting $\nu$ from \eqref{mat-eq2} into \eqref{sol-phi} we obtain the solution $\vec\varphi\in L^2(I)$ to
%\eqref{mat-eq}, provided that such a solution exists.

\section{Necessary and sufficient conditions for the existence of the $L^2(I)$ solution to the equation \eqref{mat-eq}}\label{sec-3}

In Section \ref{sec-2a} we obtained two  conditions: a) the existence of { the constant (but depending on $\psi$)  vector} 
$\vec c[\vec\psi]$, and
b) the requirement \eqref{range-2}. These  conditions are necessary for the existence of an $L^2(I)$ solution to   \eqref{mat-eq}.
In this section we first express the condition b) in terms of
{$\vec\nu= \CH^{-1} \Th^{-1}_d \left(\vec\psi-\vec c[\vec\psi]\right)$, where $\vec \psi$ is the original data, see \eqref{mat-eq2},}
%$\nu(z)$ that is closely related to the original data $\vec \psi$, see \eqref{mat-eq2},} 
and then show that the conditions a) and b) are also sufficient for the existence of the   $L^2(I)$ solution to \eqref{mat-eq}.

According to \eqref{R-3}, the   equation \eqref{range-2} can be written as follows
\be\label{R-4v}
-\frac{1}{\pi i}\int_IR^{-1}\chi\Th_o\CH\vec\varphi dz=\vec c[\vec\psi],
\ee
or, component-wise, 
\be\label{R-4s}
-\frac{1}{\pi^2 i}\int_{I_m}\frac {dz}{ R_{m+}(z)}\sum_{k\neq m}\th_{mk}\int_{I_k}\frac{\varphi_k(x) dx}{x-z}=c_m[\vec\psi],
\ee
where $c_m[\vec\psi]$ denotes the $m$-th component of  $\vec c[\vec\psi]$. Changing the order of integration and using the second equation in \eqref{H-inv[1]}, we rewrite
\eqref{R-4s} as follows
\be
\label{R-5s}
\frac{1}{\pi }\sum_{k\neq m}\th_{mk}\int_{I_k}\frac{\varphi_k(y) dy}{R_m(y)}=c_m[\vec\psi],~~~~m=1,2,\dots, n.
\ee
Substitution of \eqref{sol-phi} into \eqref{R-5s} yields
\be\label{J12}
\begin{split}
c_m[\vec\psi]&=\frac{1}{\pi }\sum_{k\neq m}\th_{mk}\int_{I_k}\frac{\nu_k(y) dy}{R_m(y)}+
\frac{1}{\pi }\sum_{k\neq m}\th_{mk}\int_{I_k}\frac{1}{R_m(y)}\int_I\frac{\nu(x)
\vec g^t(x) \Gamma^{-1}(x) \Gamma(y)\vec f(y)dx} {2\pi  i   (y-x)}dy\\
&=:J_1+J_2,
\end{split}
\ee
where $J_j$, $j=1,2$, denote the corresponding terms in \eqref{J12}. 
Recall that for $y\in I_k$, according to \eqref{f,g}, we have  $\G(y)\vec f(y)=-2\G_k(y)R_{k+}(y)$, where $\G_k$ denotes the $k$-th column of 
the matrix $\G$. Then,
changing the order of integration in $J_2$ (cf. Remark \ref{rem-non-sing}), we obtain
\be\label{J2}\begin{split}
J_2&=\frac{1}{\pi }\int_I \nu(x)
\vec {g^t}(x) \Gamma^{-1}(x)\le[
\sum_{k\neq m}\th_{mk}\int_{I_k}\frac{\Gamma(y)\vec f(y)dy}  {2\pi  i   (y-x){R_m(y)}}\ri]dx\\
&=-\frac{1}{\pi }\int_I \nu(x)
\vec {g^t}(x) \Gamma^{-1}(x)\le[
\sum_{k\neq m}\th_{mk}\int_{I_k}\frac{\Gamma_k(y) R_{k+}(y)dy}  {\pi  i   (y-x){R_m(y)}}\ri]dx.
\end{split}
\ee
%where we have used \eqref{f,g}, which implies that $\G(y)\vec f(y)=-2\G_k(y)R_{k+}(y)$ for $y\in I_k$. Here $\G_k$ denotes the $k$-th column of $\G$.
%To proceed further, we need to calculate the jumps $\D_k\G_m$ of $\G_m$ over $I_k$. According to \eqref{rhpG_mat}, \eqref{jumpG},
%\be\label{jump-col}
%\D_k\G_m=2\frac{\th_{km}R_{k+}(z)}{\th_{mm}R_{m}(z)}\G_{k-},\ k\neq m, \quad\text{and}\quad \D_m\G_m=0.
%\ee

Let $\hat\g$ be a large negatively oriented circle containing $I$. Pick a point $x\in I$. Then, using Remark \ref{rem-beh-of -col},
\be\label{cont-int}
\begin{split}
-\vec e_m&=\frac{1}{2\pi i}\oint_{\hat\g}\frac{\G_m(\z)d\z}{\z-x}=\frac{1}{2\pi i}\sum_{k\neq m}\int_{I_k}\frac{\D_k\G_m(\z)d\z}{\z-x}
-\hf\sum_{j=1}^n[\G_{m+}(x)+\G_{m-}(x)]\chi_j(x)\\
&=\frac{1}{\pi i}\sum_{k\neq m}\int_{I_k}\frac{\th_{km}R_{k+}(\z)\G_{k-}(\z)d\z}{
\th_{mm}
R_{m}(z)(\z-x)}-
\hf\sum_{j=1}^n[\G_{m+}(x)+\G_{m-}(x)]\chi_j(x).
\end{split}
\ee
Using the symmetry of $\Th$, we obtain
\be\label{the-formula}
-\frac{1}{\pi i}\sum_{k\neq m}\int_{I_k}\frac{\th_{mk}R_{k+}(\z)\G_{k-}(\z)d\z}{R_{m}(z)(\z-x)}=
\th_{mm}
\le[\vec e_m
-\hf\sum_{j=1}^n[\G_{m+}(x)+\G_{m-}(x)]\chi_j(x)
\ri].
\ee
%When $y\in I_k$, the numerator in the integrand of $J_2$ becomes $R_{k+}(y)\G_{k-}(y)$. Thus, 
Substituting
\eqref{the-formula} into \eqref{J2} we obtain
\be
J_2=\frac{
%\th_{mm}
\th_{mm}}{\pi }\int_I \nu(x)
\vec {g^t}(x) \Gamma^{-1}(x)\le[\vec e_m-\hf\sum_{j=1}^n[\G_{m+}(x)+\G_{m-}(x)]\chi_j(x) \ri]dx,
\ee
where, according to Remark \ref{rem-non-sing}, $\vec {g^t}(x) \Gamma^{-1}(x)$ does not have a jump on $I$.
Note that 
%\begin{enumerate}
 $\Gamma^{-1}(x)\G_m(x)=\vec e_m$, where $\vec e_m$ is the $m$-th standard basis vector and,
according to \eqref{f,g}, \eqref{fg-vect},
\be\label{aux-J2}
\vec {g^t}(x)\chi_j(x) \vec e_m=  
\vec {e_j^t}\Th_o\Th^{-1}_dR^{-1}(x)\vec e_m\chi_j(x)=
\left\{
\begin{array}{cc}
\frac{\th_{jm}}{\th_{mm}R_m(x)}\chi_j(x),  &  ~{\rm if}~j\neq m,\\
0,  &   ~{\rm if}~j= m.
\end{array}
\right.
\ee 
%\end{enumerate}
%where $(\vec{e^t})_j$ denotes the j-th entry of $\vec {e^t}$.
Thus, according to  \eqref{fg-vect} and \eqref{aux-J2},
%symmetry of $\Th$, we obtain
\be\label{J2fin}
\begin{split}
J_2&=\frac{\th_{mm}
%\th_{mm}
}{\pi }\le[\int_I \nu(x)
\vec {g^t}(x) [\Gamma^{-1}(x)]_mdx-\sum_{k\neq m}\th_{km}\int_{I_k}  \frac{\nu_k(x) dx}{\th_{mm}R_m(x)}  \ri]\\
&=\frac{
%\th_{mm}
\th_{mm}}{\pi }\sum_{k\neq m}\int_{I_k}    \le[\Th_o\Th^{-1}_d
%\Th^{-1}_d
R^{-1}(x)\Gamma^{-1}(x)\ri]_{km}\nu_k(x)dx-J_1,
\end{split}
\ee
where $[\cdot]_{km}$ denotes the $k,m$-th entry of the matrix. Substituting \eqref{J2fin} into \eqref{J12}, we obtain the second 
necessary condition \eqref{N2} in the form
\be\label{N2}
c_m[\vec\psi]
%\frac{1}{\pi }\sum_{k\neq m}\th_{mk}\int_{I_k}\frac{\nu_k(x) dx}{R_m(x)}+
=\frac{
\th_{mm}}
{\pi }\sum_{k\neq m}\int_{I_k}    \le[\Th_o
\Th^{-1}_d
R^{-1}(x)\Gamma^{-1}(x)\ri]_{km}\nu_k(x)dx,
%=\frac{1}{\pi }\sum_{k\neq m}\int_{I_k}    \le[\Th_o
%\Th^{-1}_d
%R^{-1}(x)\Gamma^{-1}(x)\ri]_{km}\nu_k(x)dx
%=\frac{\th_{mm}}{\pi }\sum_{k\neq m}\int_{I_k} \le(\frac{\th_{km}}{\th_{mm}R_m(x)}+
%\hf\le[\Th_o\Th^{-1}_dR^{-1}(x)\Gamma^{-1}(x)\ri]_{km}\ri)\nu_k(x)dx, 
~~~~\qquad m=1,\dots, n,
\ee
where, according to Remark \ref{rem-non-sing} and \eqref{J2fin}, the right hand side of \eqref{N2} is independent of the choice 
of $\G^{-1}_+$ or $\G^{-1}_-$.

\bth\label{theo-main}
The system \eqref{mat-eq} with a positive-definite symmetric matrix $\Th$ has a solution $\varphi\in L^2(I)$ if and only if the following
two conditions are satisfied:
\begin{enumerate}
 \item There exists a constant vector $\vec c[\vec \psi]\in \mathbb R^n$ such that $\vec\psi-\vec c[\vec\psi]\in \Rscr$;
 \item All the components $c_m[\vec\psi]$ of the vector $\vec c[\vec \psi]$ satisfy equations \eqref{N2},
 where 
 \be
 \vec \nu=\text{\rm Col}(\nu_1,\dots,\nu_n)=\CH^{-1} \Th^{-1}_d \left(\vec\psi-\vec c[\vec\psi]\right).
 \ee
\end{enumerate}
The solution $\varphi\in L^2(I)$, if exists, is unique and is given by \eqref{sol-phi}.
\end{theorem}

\begin{proof}
The necessary part follows from Lemmas \ref{lem-N1}, \ref{lem-N2}.  Now assume 
that  the vector $\vec c[\vec \psi]$ exists.
To prove the sufficient part, notice that equations
\eqref{mat-eq1} and \eqref{mat-eq3} are equivalent if and only  the condition \eqref{range-2} holds.
Then, if the condition \eqref{range-2} holds, it is sufficient to show that \eqref{mat-eq3}  has a solution 
$\varphi\in L^2(I)$. Existence of such solution follows from the fact that $\l=1$ is not an eigenvalue of $\wh K$.
That completes our argument.
\end{proof} 

In the particular case $n=2$ we have
\be\label{cn=2}
c_1[\vec\psi]=\frac{\th_{21}} \pi \int_{I_2} \frac{\G_{22}(x)\nu_2(x)dx}{R_1(x)},~~~~~
c_2[\vec\psi]=\frac{\th_{12}} \pi \int_{I_1} \frac{\G_{11}(x)\nu_1(x)dx}{R_2(x)}.
\ee

\bc\label{cor-c=0}
In the particular case when $\vec\psi\in\Rscr$ the equations \eqref{N2} from Theorem \ref{theo-main} become
\be\label{eqc=0}
\sum_{k\neq m}\int_{I_k}    \le[\Th^{-1}_{d}\Th_o\Th^{-1}_d
R^{-1}\Gamma^{-1}\ri]_{km}\CH^{-1}_k[\psi_k]dx =0  
\ee
for all $m=1,\dots,n$.
\ec

\br \label{rem-L1}
In the case $R^{-1}\vec \psi \in L^1(I)$, according to \eqref{inv_H} and \eqref{H-inv[1]}, 
the equations \eqref{N2} from Theorem \ref{theo-main} become
\be\label{eqL1}
%c_m[\vec\psi]=-
i\int_{I_m} \frac{\psi_m}{R_{m+}}dx+{\th_{mm}}\sum_{k\neq m}\int_{I_k}    \le[R\Th^{-1}_d\Th_o\Th^{-1}_d
R^{-1}\Gamma^{-1}\ri]_{km}\CH_k\le[\frac{\psi_k}{R_k}\ri]dx  =0
\ee
for all $m=1,\dots,n$.
\er

The 
 range condition \eqref{N2} can be viewed as a null-space of some unbounded linear functional in $L^2(I)$ with an everywhere dense domain.  
 It contains  three main components: the linear map $\vec\psi\to   c_m[\vec\psi]$, $n$ linear maps $\psi_j \to\nu_j$, $j=1,\dots,n$, and the sum of integrals with the weights $[\cdot]_{km}$. The first  map is clearly unbounded. As we show in Lemma \ref{rem-cont}, the weights $[\cdot]_{km}$ in \eqref{N2} are functions analytic in $x$ at least continuous in all $\al_j, \bt_j$ when $x\in I_k$. This means that the $n$-dimensional part does not create any additional complications, and the dependence of the range condition on the endpoints becomes essentially the same as in the one-interval case. The latter is outside the scope of this paper.
%Suppose all the components of $\vec \psi$ are sufficiently nice. Then the map $\vec\psi\to \vec c[\vec\psi]$ is given by 
%\be\label{cond}
%c_m=(??)\int_{I_m} \frac{\psi_m}{R_{m+}}dx,\ m=1,\dots,n.
%\ee
%Since $R_{m+}$ is not in $L^2(I_m)$, it is clear that the functional $\vec\psi\to \vec c[\vec\psi]$ acting on $L^2(I)$ is not bounded. Likewise, the range of each Finite Hilbert transform $\chi_m\CH_m:\,L^2(I_m)\to L^2(I_m)$ is everywhere dense, so the maps $\psi_m\to\nu_m$ acting between the spaces $L^2(I_m)\to L^2(I_m)$ are not bounded either. This complicates the analysis of the dependence of the two maps, $\vec\psi\to \vec c[\vec\psi]$ and $\psi_m\to\nu_m$, on the endpoints $\al_j$, $\bt_j$. Nevertheless, this analysis is essentially reduced to the analysis in the case of the conventional FHT on an interval (cf. \eqref{Hilb}) and is outside the scope of this paper. The remaining part, which arises specifically due to the presence of the system of equations (namely, the integral on the righ-hand side of \eqref{N2}) is well-behaved. As we show in ???, the coefficients $[\cdot]_{km}$ in \eqref{N2} are good (to be made precise) functions of $x,\al_m$, and $\bt_m$. This means, that the $n$-dimensional part does not create any additional complications, and the stability analysis becomes essentially the same as in the one-interval case.

\bl\label{cont-Gam}
The solution $\G(z)$ to the RHP \ref{RHPGamma} locally analytically depend  on $\a_j,\bt_j$, $j=1,\dots,n$ for any $z\in \bar\C$ except the endpoints
$\a_j,\bt_j$, $j=1,\dots,n$, where it may have the square root $\sqrt{z-\a}$ type singularities. The same statement is true for  $\G^{-1}(z)$.
 \el
 %{\color{ blue}
 \begin{proof}
 % The statement follows from \eqref{N2} provided we can show that that $\G(z)$ and, thus,  $\G^{-1}(z)$ smoothly depend In fact, we will
 To show that 
%
 %\subsection{Analytic dependence on the endpoints}
 %We now show that the solution 
 $\G(z)$ depends analytically on the endpoints, we define ``local solutions'' near the intervals. For each of the interval $I_\ell$,
 define 
 \be
 P_\ell (z):= \1 + {2}\vec e_\ell\int_{I_\ell }  \frac { \le[ \frac{\theta_{\ell 1 } R_{\ell +}(w)}{\theta_{11} R_1(w)} , \dots,0,\dots \frac{\theta_{\ell n} R_{\ell +}(w)}{\theta_{nn} R_n(w)} \ri]dw}{(w-z)2\pi i},
 %\1 + \vec e_\ell\oint_{B_\ell }  \frac { \le[ \frac{\theta_{\ell 1 } R_{\ell }(w)}{\theta_{11} R_1(w)} , \dots,0,\dots \frac{\theta_{\ell n} R_{\ell }(w)}{\theta_{nn} R_n(w)} \ri]dw}{(w-z)2\pi i},
 \ee 
 where 
 %$B_\ell$ is a negatively oriented contour around $I_\ell$ that does not intersect any other interval $I_k$, $k\neq \ell$ and		
 %${\bf e}_\ell$ denotes the standard elementary column vector and
 the $0$ in the integrand is in the $\ell$-th position. 
 By the Plemelj--Sokhotski formula if follows that 
 \be
 P_{\ell+} (z)  = P_{\ell- }(z) V(z,1),\ \ \ \ \ z\in I_\ell, 
 \ee
 where $V$ is defined by \eqref{jumpG}. Indeed,
  \be
 P_{\ell+} (z)  - P_{\ell- }(z)={2}\vec e_\ell \le[ \frac{\theta_{\ell 1 } R_{\ell +}(z)}{\theta_{11} R_1(z)} , \dots,0,\dots \frac{\theta_{\ell n} R_{\ell +}(z)}{\theta_{nn} R_n(z)} \ri]= \vec e_\ell \le[ \frac{\theta_{\ell 1 } R_{\ell +}(z)}{\theta_{11} R_1(z)} , \dots,0,\dots \frac{\theta_{\ell n} R_{\ell +}(z)}{\theta_{nn} R_n(z)} \ri]P_{\ell -}(z)
 \ee
 Of course $P_\ell $ fails to solve the jump condition
 \eqref{rhpG_mat} on the remaining intervals $I_k,\ k\neq \ell$. Note also that $\det P_\ell (z) \equiv 1$ and hence local solutions are analytically invertible.

Let now $\mathbb D_\ell,\ \  \ell=1,\dots, n$ be mututally disjoint open disks (or regions) such that $I_\ell \subset \mathbb D_\ell$.
We also observe that the local solution $P_\ell(z)$ depends analytically on the endpoints $\a_j,\bt_j$, $j=1,\dots,n$, for any $z\in\mathbb D_\ell,$ except at the endpoints $z=\a_\ell$ and $z=\bt_\ell$ of $I_\ell$,
 where it may have the square root $\sqrt{z-\a}$ type singularities. 
Define 
\be
Q(z):= \le\{
\begin{array}{cl}
\G(z) & z\not\in \bigcup _\ell \mathbb D_\ell\\
\G(z) P_\ell^{-1}(z) & z\in \mathbb D_\ell, \qquad \ell=1,\dots, n.
\end{array}
\ri.
\ee
This new matrix $Q(z)$ has only jumps on the boundaries $\pa \mathbb D_\ell$  and it satisfies the RHP
\be
\label {RHP_R}
Q_+ (z)= Q_-(z) P_{\ell}(z) ,\qquad  z\in \pa\mathbb D_\ell\qquad  {\rm and}~~ 
Q(\infty)=\1,
\ee
where the orientation of the boundary is clockwise. If we now move slightly the endpoints (while keeping the disks $\mathbb D_\ell$ fixed), the jump-matrices of the RHP \eqref{RHP_R} change analytically; by the analytic Fredholm theorem, so does the solution $Q(z)$.
This, in turn implies the statement of the lemma for $\G(z)$. Since $\det \G (z) \equiv 1$, the same statement is true for $\G^{-1}(z)$. 

Alternatively, we can prove the analyticity of $Q(z)$ by differentiating both sides of \eqref{RHP_R} in $\pa=\frac{\pa}{\pa\a}$, where 
$\a$ denotes one of the endpoints. Then the differentiated  RHP \eqref{RHP_R} becomes
\be
\label {RHP_R_diff}
\pa Q_+ (z)= \pa Q_-(z) P_{\ell}(z)+  Q_-(z) \pa P_{\ell}(z),\qquad  z\in \pa\mathbb D_\ell\qquad  {\rm and}~~ 
Q(\infty)=0,
\ee
which is satisfied by
\be\label{diff-Q}
\pa Q= \sum _{\ell=1}^nC_{\pa\mathbb D_\ell} (Q_-\pa  P_{\ell} Q_+^{-1})Q=\sum _{\ell=1}^nC_{\pa\mathbb D_\ell} (\G\pa  P_{\ell} \G^{-1})Q,
\ee
where $C_{\g}$ denotes the Cauchy operator over the oriented contour $\g$ and we used the fact that $ P_{\ell}^{-1}\pa  P_{\ell}=\pa  P_{\ell}$.
(Note that $\G\pa  P_{\ell} \G^{-1}$ is smooth (analytic) on $\pa\mathbb D_\ell$ for any $\ell$). Thus, $Q(z)$ is analytic in the endpoints 
for all $z\in \bar\C$.
\end{proof}

\bl\label{rem-cont}
Let us fix some index $m\leq n$.
The weights $\le[\Th_o
\Th^{-1}_d
R^{-1}(x)\Gamma^{-1}(x)\ri]_{km}$ from \eqref{N2} are analytic in $z$ on $I_k$ and continuous in all
the points $\a_j,\bt_j$, $j=1,\dots,n$ for all $z\in I_k$.
%\textcolor{red}{Need to explain why $\Gamma^{-1}$ depends smoothly on $\al_j,\bt_j$.}
 \el
 %{\color{ blue}
 \begin{proof}
 Note that the $k$th row of the matrix $\Th_o\Th^{-1}_dR^{-1}$ contains all the radicals $R_j$ in the denominator except for 
 $R_k$, since ${\rm diag}\Th_o=0$. Then, according to Remark \ref{rem-beh-of -col}, the functions  
 $\le[\Th_o
\Th^{-1}_d
R^{-1}(x)\Gamma^{-1}_\pm(x)\ri]_{km}$ are analytic on the interior of $I_k$. But since these functions coincide on $I_k$, see the comment
under \eqref{N2}, we conclude that $\le[\Th_o
\Th^{-1}_d
R^{-1}(x)\Gamma^{-1}(x)\ri]_{km}$ is analytic in $z$ in some neighborhood of $I_k$. The use of Lemma \ref{cont-Gam}  completes the proof. 
 \end{proof}

\section{Invertibility of the operator $\Id-\wh K$ and uniqueness of solution to \eqref{mat-eq}}\label{sec-2}

In this section we consider the equation 
\be\label{mat-eq-hom}
\chi \Th \CH \vec f = \vec c,
\ee
where $\Th$ is a symmetric positive definite matrix, {and $\vec c$ is a constant vector that may depend on $\vec f$.} Theorem 
\ref{theo-inv}  below shows that the equation \eqref{mat-eq-hom} with any $\vec c$ does not have a nontrivial solution $\vec f\in L^2(I)$.
{Thus, as it was mentioned in the proof of Lemma \ref{lem-specK}, $\l=1$ is not an eigenvalue of  the operator $\Id-\wh K$,
and  the operator $\Id-\wh K$ is invertible. Moreover, choosing  in \eqref{mat-eq-hom} $\vec c=0$, we show the uniqueness of an $L^2$ solution
to \eqref{mat-eq},  provided that such a solution  exists. }

\br \label{rem-diag-1}
Here and in the rest of this section, without any loss of generality, we can assume $\text{diag~}\Th=\1$. Indeed, denote ${D}= {\rm diag}~\Th$. 
Then, because of symmetry and positive definiteness, 
all the diagonal  entries $D$ are positive. Then $\Th=D^\hf\tilde\Th D^\hf$, where $\tilde\Th$ is  positive definite, symmetric
and $\text{diag~}\tilde\Th=\1$.
Changing variables $\vec f \mapsto D^\hf \vec f$,  $\vec c \mapsto D^{-\hf} \vec c$ we obtain \eqref{mat-eq-hom} with
the matrix $\Th$ replaced by $\tilde\Th$.
\er

According to Remark \ref{rem-diag-1}, we have  $\Th_o=\Th -\1$ for the off-diagonal part of $\Th$.
%The diagonal part of $\Th$ sometimes will be denoted $\Th_d$. As it was mentioned in Section \ref{sec-intro}, we can put $\Th_d=\1$. 

\bl \label{lem-f_j_glob}
If $\vec f\in L^2(I)$ and $\vec f$ solves \eqref{mat-eq-hom}, then 
\be \label{f_j-glob}
\vec f(z)=R(z)\vec g(z),
\ee 
where  each component $g_j(z)$ of $\vec g(z)$ is analytic in $\bar\C\setminus\le(\cup_{k\neq j}I_k\ri)$.
\el

\begin{proof}
Let $f_j(z)$, $~j=1,2\dots, n$, denote the entries of $\vec f(z)$. According to \eqref{mat-eq-hom}, $\chi\Th_o \CH\vec f-\vec c$ must be in the range of $\chi\CH$. So, we can write the $j$-th component
%Since {$\Th$ is positive definite,} $\th_{jj}\neq 0$, {so,} assuming $\vec f$ to be solution of   \eqref{mat-eq-hom}, we obtain 
\be\label{fj-eq}
f_j=-\sum_{k\neq j}{\th_{jk}}\CH^{-1}_j\CH_k f_k{+\CH^{-1}_jc_j}
\ee
on $I_j$, $j=1,2\dots,n$. We now use \eqref{H-inv[1]} and Lemma \ref{lem-H-1H} to get
\be\label{fj-eq2}
f_j(z)=\frac{R_j(z)}{\pi i}\sum_{k\neq j}{\th_{jk}}\int_{I_k}\frac{f_k(x)dx}{(x-z)R_j(x)}= 
-iR_j(z)\sum_{k\neq j}\th_{jk}\CH_k\le[\frac{f_k}{R_j}\ri](z)=R_j(z)g_j(z),
\ee
where $z\in I_j$ and  $R_j(z)$ is taken on the positive side (the upper shore) of $I_j$. It follows from \eqref{fj-eq2} that 
$g_j$ is analytic 
 in $\bar\C\setminus\le(\cup_{k\neq j}I_k\ri)$.
 \end{proof}

Pick any $s$, $1/2<s<1$, and consider the Hilbert space norm
\be\label{s-norm}
\Vert f\Vert_s^2:=\int_{\mathbb R} |\tilde f(\xi)|^2(1+|\xi|^2)^s d\xi,\ f\in C_0^{\infty}(\mathbb R),
\ee
where $\tilde f$ is the Fourier transform of $f$:
\be\label{ft}
\tilde f(\xi)=\int_{\mathbb R} f(x)e^{ix\xi}dx.
\ee
Let $C_0^{\infty}(I)$ be the space of smooth functions which vanish at the endpoints of each $I_j$. Define the Hilbert space $\dot H_s(I)$ as the closure of $C_0^{\infty}(I)$ in the following norm
\be\label{final_norm}
\vvvert f\vvvert:=\left(\sum_j \Vert f_j\Vert_s^2\right)^{1/2}.
\ee
Here and in what follows, given a function $f\in \dot H_s(I)$, $f_j$ denotes the restriction of $f$ onto the interval $I_j$ and extended by zero outside $I_j$. Likewise, a collection of $f_j$ (with the appropriate properties) determines a function $f\in \dot H_s(I)$. 

Note two facts. (i) Any function $f$, whose pieces $f_j$ satisfy \eqref{fj-eq2}, belongs to $\dot H_s(I)$. This follows from the fact that due to the $\sqrt{x_+}$-type singularity, the Fourier transform of $f$ satisfies $\tilde f(\xi)=O(\xi^{-3/2})$, $|\xi|\to\infty$; (ii) Any $f\in \dot H_s(I)$ is  continuous (as is known, $H_s(\mathbb R)\subset C(\mathbb R)$, see \cite{hor},  Corollary 7.9.4). 

Consider the system of equations \eqref{mat-eq-hom}, which we write as
\be\label{eq1}
\frac1\pi\sum_k\theta_{jk}\int_{I_k}\frac{f_k(x)}{x-y}dx=c_j, \ y\in I_j,\ 1\leq j\leq n,
\ee
where $f\in\dot H_s(I)$. Observe that (1) the multi-interval Hilbert transform $\CH: \dot H_s(I)\to H_s(\mathbb R)$ is bounded. This follows from the fact that each finite Hilbert transform in \eqref{eq1} is just the multiplication with $-i\text{sgn}(\xi)$ in the Fourier domain; and (2) $H_s(\mathbb R)\subset C(\mathbb R)$. Thus the above equation makes sense pointwise.

\begin{theorem}\label{theo-inv}
If $f\in\dot H_s(I)$ solves the system \eqref{eq1} {with a positive definite symmetric matrix $\Th=(\th_{jk})$} 
for some constants $c_j$, then all $c_j=0$ and $f\equiv0$.
\end{theorem}

\begin{proof}
Pick any $\phi\in C_0^{\infty}(I)$, multiply the $j$-th equation by $\phi_j'$, integrate over $I_j$, and add the results. Clearly, we get
\be\label{eq1-phi}
\sum_j \int_{I_j}\phi_j'(y)\left\{\frac1\pi\sum_k\theta_{jk}\int_{I_k}\frac{f_k(x)}{x-y}dx\right\}dy=0.
\ee

In view of \eqref{eq1-phi}, we define the following bilinear form:
\be\label{eq3}
\begin{split}
J(f,g):&=\sum_j \int_{I_j}g_j'(y)\left\{\frac1\pi\sum_k\theta_{jk}\int_{I_k}\frac{f_k(x)}{x-y}dx\right\}dy\\
&=\frac1\pi\sum_{j,k}\theta_{jk}\int_{I_j}\int_{I_k} g_j'(y)f_k'(x) \ln\le(\frac 1 {|x-y|}\ri)dxdy,\ f,g\in C_0^{\infty}(I).
\end{split}
\ee
Recall that $f_j$'s are the pieces that make up $f$ (and the same for $g$). With some abuse of notation, we write $J(f):=J(f,f)$. First, $J(f,g)=J(g,f)$. Also, it is easy to see that $J$ is continuous, positive definite, and strictly convex on $\dot H_s(I)$. Indeed, extending $f_k$ and $g_j$ by zero outside $I_k$ and $I_j$, respectively, and using that the Fourier transform preserves dot products, we get from \eqref{eq3}
\be\label{eq4}\begin{split}
J(f,g)&=\frac1{2\pi}\sum_{j,k}\theta_{jk}\int_{\mathbb R}[-i\text{sgn}(\xi)\tilde f_k(\xi)] \overline{[-i\xi\tilde g_j(\xi)]}d\xi\\
&=\frac1{2\pi}\sum_{j,k}\theta_{jk}\int_{\mathbb R}|\xi|\tilde f_k(\xi)\overline{\tilde g_j(\xi)}d\xi.
\end{split}
\ee
Since $s > 1/2$, from \eqref{s-norm} and \eqref{final_norm} we have
\be\label{J-ineq}\begin{split}
|J(f,g)| &\leq \frac1{2\pi}\sum_{j,k}|\theta_{jk}|\int_{\mathbb R}\sqrt{|\xi|}|\tilde f_k(\xi)|\cdot\sqrt{|\xi|}|\tilde g_j(\xi)|d\xi\\
&\leq \frac1{2\pi}\sum_{j,k}|\theta_{jk}|\left(\int_{\mathbb R}|\xi||\tilde f_k(\xi)|^2d\xi\right)^{1/2} \left(\int_{\mathbb R} |\xi||\tilde g_j(\xi)|^2d\xi\right)^{1/2}\\
&\leq c \sum_{j,k} \Vert f_k\Vert_s \Vert g_j\Vert_s \leq cn\, \vvvert f \vvvert \cdot \vvvert g \vvvert
\end{split}
\ee
for some $c>0$. Thus, $J$ is continuous and extends to all of $\dot H_s(I)\times\dot H_s(I)$.

Considering $\tilde f$ as a vector, $\tilde f(\xi)=(\tilde f_1(\xi),\dots,\tilde f_n(\xi))$, we have
\be\label{pos-def}
J(f)=\frac1{2\pi}\int_{\mathbb R}|\xi|(\Theta \tilde f(\xi),\tilde f(\xi))d\xi.
\ee
Here $(\cdot,\cdot)$ denotes the usual dot product in $\mathbb C^n$. The remaining assertions, convexity and positive definiteness, are now obvious because $\Theta$ is positive-definite. In particular, \eqref{pos-def} proves that $J(f+tg)$ is a parabola with respect to $t$ for any $g\in \dot H_s(I)$, $g\not\equiv0$.

Consider the variation $J(f+t\phi)$, $f,\phi\in \dot H_s(I)$. We have (cf. \eqref{eq3}):
\be\begin{split}\label{variation}
\frac{d}{dt}\left.J(f+t\phi)\right|_{t=0}
=2J(f,\phi)=2\sum_j \int_{I_j}\phi_j'(y)\left\{\frac1\pi\sum_k\theta_{jk}\int_{I_k}\frac{f_k(x)}{x-y}dx\right\}dy.
\end{split}
\ee
Suppose now that $f$ satisfies \eqref{eq1} and, consequently, $f$ satisfies \eqref{eq1-phi} for all $\phi\in C_0^{\infty}(I)$. Then the right-hand side of \eqref{variation} equals zero, and $f$ is a critical point of $J(f)$. By convexity, the only critical point is $f\equiv0$. Hence, all $c_j$ are zero, and any solution $f\in \dot H_s(I)$ of \eqref{eq1} is necessarily trivial.
\end{proof}

  \section{Other cases for the  matrix $\Th$}\label{sec-other}
  
  We first observe that in the previous sections the positive-definiteness of matrix $\Th$ was used only in two places: (i) 
  to show that $\Th_d$ is invertible, and (ii) to show that  $\lambda =1$ is not an eigenvalue of   the  operator $\wh  K$. Thus, we obtain the following corollary.
  
\bc  \label{cor-not-pos-def}
   If matrix $\Th$ is symmetric with invertible diagonal  part $\Th_d$,  then Theorem \ref{theo-inv-pos-def} (inversion formula) and
  Theorem \ref{theo-main} (range conditions) are still valid provided  that   $\lambda =1$ is not an eigenvalue of   the  operator $\wh  K$
  defined  by \eqref{operKn}.
 \ec
 
 Now, the symmetry of $\Th$ is used only in the calculation of the range condition \eqref{N2}, namely in the transition from \eqref{cont-int} to \eqref{the-formula}. Thus, we can state 
 %We summarize this case in 
 the following corollary. 
 
 \bc  \label{cor-not-sym}
 If $\Th_d={\rm diag} \Th$ is invertible and if $\lambda =1$ is not an eigenvalue of   the  operator $\wh  K$
  defined  by \eqref{operKn}, then Theorem \ref{theo-inv-pos-def}  and
  Theorem \ref{theo-main}  are still valid provided  that the range condition 
  \eqref{N2} in   Theorem \ref{theo-main} is replaced
  with \eqref{J12}, where $\vec f,\vec g$ are defined by \eqref{f,g}.
 \ec

{Next we consider} the remaining case of a non-invertible matrix $\Th_d$. We start with an example of $n=2$. In this example we assume $\th_{11}=0$,
 but exclude the trivial case of $\Th=0$.
{T}he first equation of \eqref{mat-eq} (on $I_1$) becomes $\CH_2\varphi_2=\frac{\psi_1}{\th_{12}}$, provided that $\th_{12}\neq 0$.
 Since $\CH_2\varphi_2$ is analytic in $\bar \C\setminus I_2$, we conclude that $\psi_1$ can be analytically continued in  $\bar \C\setminus I_2$, 
 $\psi_1(\infty)=0$, and the jump $\D_2\psi_1$ of $\psi_1$ over $I_2$ must be in $L^2(I_2)$. These are  the range conditions for $\psi_1$.
 Moreover, according to the Plemelj-Sokhotski formula, $\varphi_2=\frac{\D_2\psi_1}{2i\th_{12}}$.
 
 The second equation (on $I_2$) now has the form 
 \be\label{eq-phi1}
 \th_{21}\CH_1\varphi_1=\psi_2-\frac{\th_{22}}{2i\th_{12}}\CH_2(\D_2\psi_1)=:\tilde\psi_2.
 %= \psi_2-\frac{\th_{22}}{2\th_{12}}(\psi_{1+}+\psi_{1-}) (\red{??}
 \ee
 Assuming $\th_{21}\neq 0$, we can use the previous argument to obtain the range condition for the adjusted right-hand side $\tilde\psi_2$: 
 the function $\tilde\psi_2$  can be analytically continued  in $\bar\C\setminus I_1$, it 
 attains zero at infinity
 and also has an $L^2(I_1)$ jump on $I_1$. In this case $\varphi_1=\frac{1}{2i\th_{21}}\D_2\le[ \tilde\psi_2 \ri]$.
 %\psi_2 -\frac{\th_{22}}{2i\th_{12}}\CH_2(\D_2\psi_1)\ri]$.
 
 If $\th_{21}=0$, then $\varphi_1$ can be any $L^2(I_1)$ function, and $\psi_2=\frac{\th_{22}}{2i\th_{12}}\CH_2(\D_2\psi_1)$ on $I_2$ becomes
 the range condition for $\psi_2$. 
 
 Consider the remaining case of $\th_{12}=0$. Now the first equation (on $I_1$) becomes trivial.
 Then $\psi_1\equiv 0$ is the range condition for $\psi_1$. The second equation (on $I_2$) becomes
 \be\label{eq-phi2}
 \th_{22}\CH_2\varphi_2= \psi_2-\th_{21}\CH_1\varphi_1.
 \ee
 If  $\th_{22}=0$, then  $\th_{21}\neq 0$, and we can repeat our previous arguments to obtain  $\varphi_1=\frac{\D_1\psi_2}{2i\th_{21}}$, which leads to the range conditions for $\psi_2$: it can be analytically continued in  $\bar \C\setminus I_1$, 
 $\psi_2(\infty)=0$ and the jump $\D_1\psi_2$ of $\psi_2$ over $I_1$ must be in $L^2(I_2)$.
 
 Otherwise, if  $\th_{22}\neq 0$, the right hand side of \eqref{eq-phi2} must be  in the range of $\CH_2$ {restricted to $I_2$}. So, the range conditions for $\psi_2$ are: (i) there exists $c\in \R$ such that $\psi_2-c$
 is in the range of the finite Hilbert transform $\chi_2\CH_2: L^2(I_2)\ra L^2(I_2)$, and (ii) 
 $c=-\frac{\th_{21}}{\pi i}\int_{I_2}R_2^{-1}\CH_1\varphi_1 dz$ for some $\varphi_1\in  L^2(I_1)$. 
 The latter condition can always be satisfied since we are free to choose any  $\varphi_1\in L^2(I_1)$; thus
 the range condition for $\psi_2$  consists only of the condition (i). If it is satisfied, then we solve \eqref{eq-phi2} for 
 $\varphi_2$ and, thus, obtain $\vec \varphi$.
 Of course,  this solution is not unique, as the condition (ii) does not determine $\varphi_1$ uniquely. 
 That concludes the case of a $2\times 2$ matrix $\Th$ with a non-invertible diagonal. 
 
Based on the example of  a $2\times 2$ matrix $\Th$ with a non-invertible diagonal we can briefly outline the general $n\times n$ case. We say that {\it a row $k$ is connected with a row  $j$} if there is a set of distinct  indices $j,m_1,\dots,m_p,k$ such that $\th_{jm_1}\th_{m_1m_2}\cdots \th_{m_pk}\neq 0$. A row $k$ is called  {\it degenerate} if either $\th_{kk}=0$ or if row  $k$ is connected with some row $j$  such that $\th_{jj}=0$. A matrix $\Th$ with non-invertible $\Th_d$ is called {\it irreducible} if all the rows $1,\dots,n$ are degenerate; otherwise, $\Th$ is called {\it reducible}. In particular, in the above $2\times 2$ example with $\th_{11}=0$, the matrix $\Th$ is irreducible if either $\th_{12}\neq 0$ or  $\th_{22}=0$;  otherwise $\Th$ is reducible.
 
 In the case of an irreducible matrix $\Th$, all the components $\varphi_j$ of the solution vector $\vec \varphi$ can be found as jumps
 over $I_j$ of the analytic continuations of the corresponding adjusted right hand 
 sides, see the example of a  $2\times 2$ matrix $\Th$ described above. {The function $\tilde \psi_2$  in \eqref{eq-phi1} is an example of an adjusted right-hand side.}
 %linear combinations of the components $\psi_k$ of the right hand  side vector $\vec\psi$ and their successive Hilbert transforms, 
 The range conditions over $I_j$ in this case consist of: (i) certain analytic properties of the adjusted $\tilde\psi_j$, and (ii) if some $\varphi_j$ can be represented as  a jump of different analytic continuations,  then all such analytic continuations should have the same jumps over  $I_j$.
 
 In the case when the matrix $\Theta$ is reducible, we can bring it to a block-triangular form by interchanging rows and  interchanging  columns.
 Such moves do not change the degeneracy of any particular row. 
 %(more precisely, a column originally associated with the index)
 Let us move all the degenerate rows and the corresponding  columns of $\Th$ 
 %that correspond to the degenerate indices in
 to the positions $1,\cdots,m$, where $m<n$. 
 Then we obtain a lower triangular block matrix 
\be
\tilde\Th=\le[\begin{array}{cc}
 A & 0\\
B & C\\
\end{array}\ri]
\ee
where the $m\times m$ matrix $A$ is irreducible, and the $(n-m)\times(n-m)$ matrix $C$ has an invertible ${\rm diag} C$. Indeed, consider a column
 $\Th_j$ of $\Th$ with $j>m$. If any $\th_{kj}\neq 0$ with $k\leq m$, then the row $j$ is degenerate, which contradicts the assumption.
 Finding $\varphi_1,\dots, \varphi_m$ by solving  the first  $m$ equations with the corresponding irreducible  matrix $A$,
 we reduce the size of the original problem from $n\times n$ to $(n-m)\times (n-m)$. As it was mentioned above, the matrix $C$ of 
 the reduced system  has  invertible ${\rm diag} C$, so it satisfies the requirements of Corollary \ref{cor-not-sym}. Note that
 the right hand side of the reduced system depends on  the already obtained  $\varphi_1,\dots, \varphi_m$.

\section{The case of uniform interaction matrix $\Th$}\label{sec-unif}

%\textcolor{red}{In this section, we need to replace $f$ with $\varphi$, and $g$ with $\psi$ to be consistent. Alternatively, we can say that %notations are slightly different (vectors versus scalars), and that is why we use different letters in this section.} 

In the particular case when all entries $\theta_{jk}=1$, the matrix  $\Th=(\th_{jk})$ is not positive definite, so that the invertibility
of the operator $\Id-\wh K$, given by \eqref{operKn}, cannot be guaranteed using the methods of the previous sections.  However, this problem can be reduced 
 to the inversion of the multi-interval finite Hilbert transform:
\begin{equation}\label{hilb-eq}
(\CH f)(z)=g(z),\,z\in I,\ f,g\in L^2(I),
\end{equation}
where
\begin{equation}\label{hilb-orig}
\CH:\,L^2(I)\to L^2(I),\ (\CH f)(z)=\frac1\pi \int_{I}\frac{f(x)}{x-z}dx.
\end{equation}

{
\br \label{new-H}
In this section the operator $\CH$ is defined by \eqref{hilb-orig} instead of \eqref{def-Hj}, as is the case in the rest of the paper. 
{We also use here the notations $f$ and $g$ instead of $\varphi$ and $\psi$, respectively.}     
\er}

Define
\be\label{oepols}
\begin{split}
\bto(z)&= \prod_{j=1}^{n}(z-\a_{j}),\quad 
\bte(z)= \prod_{j=1}^{n}(z-\bt_{j}), \quad \bet(z)=\bte(z)/\bto(z),\\ 
\phi(z)&=\Re \ln \bet(z).
\end{split}
\ee
Note that 
\be\label{arg-ln}
\arg \bet(z)=\pi,\  \ln \bet(z)=\phi(z)+i\pi,\ z\in I. 
\ee
Using \eqref{arg-ln}, we get from \eqref{oepols}:
\be\label{der-2}
\phi'(x)=\frac{Q(x)}{\bto(x)\bte(x)},\quad
Q(x):=\bte'(x)\bto(x)-\bte(x)\bto'(x).
\ee
It is known that $Q(x)$ is positive and bounded away from zero on $I$ (cf. \cite{kbt18}) 
{and, therefore, $\phi(x)$ is monotonic  and, thus, invertible on each interval $I_j$. Moreover,
it is straightforward to see that the range of $\phi(x)$  on each interval $I_j$ is $\R$. } 
Now we calculate using \eqref{arg-ln} again:
\be \label{cos_sinh_1}
\begin{split}
2\sinh\le( \frac{\phi(x)-\phi(z)}2\ri)=\frac{(x-z)\sum_{i,j=1}^n B_{ij}z^{i-1}x^{j-1}}
{\sqrt{\prod_{j=1}^{n}(x-\a_j)(x-\bt_j)(z-\a_j)(z-\bt_j)}}, 
\end{split}
\ee
where $B:=B(\bte,\bto)=(B_{ij})$ is the B\'{e}zout matrix of the polynomials $\bte(z),\bto(z)$.
Note that $\prod_{j=1}^{n}(x-\a_j)(x-\bt_j)(z-\a_j)(z-\bt_j)>0$
%\prod_{j=1}^{2n}(x-b_j)(z-b_j)>0$ 
for any $z,x$ in the interior of $I$, and the square root in \eqref{cos_sinh_1} is computed according to the rule
\be\label{sgnsgn}
\sqrt{\prod_{j=1}^{n}(x-\a_j)(x-\bt_j)(z-\a_j)(z-\bt_j)}=-\text{sgn} \bto(x)\text{sgn} \bto(z)\prod_{j=1}^{n}|(x-\a_j)(x-\bt_j)(z-\a_j)(z-\bt_j)|^\hf.
\ee
Since $B$ is symmetric,
\be\label{bezout}
B=\Omega^t \text{diag}(\rho_1,\dots,\rho_n)\Omega,
\ee
for an orthogonal matrix $\Omega$. Then 
\be \label{bas_poly11}
\sum_{i,j=1}^n B_{ij}z^{i-1}x^{j-1}=\vec z_n^t B \vec x_n=\vec z_n^t \Omega^t \text{diag}(\rho_1,\dots,\rho_n)\Omega\vec x_n=\sum_{j=1}^n \rho_j P_j(z)P_j(x),
\ee
where $\vec z_n^t=(1,z,\dots z^{n-1}), \vec x_n^t=(1,x,\dots x^{n-1})$ and $(P_1(z),\dots,P_n(z))=\vec z_n^t \Omega^t$.

Introduce the isometry of the two spaces
\be\label{tr-three}
\begin{split}
&T:\ L^2(I)\to L_n^2(\mathbb R),\\
&\check\bff (t):=(T f)(t):=\sqrt2\left(\left.\frac{{\text{sgn}(\bto(x))}f(x)}{\sqrt{|\phi'(x)|}}\right|_{x=\phi_1^{-1}(2t)},\dots,\left.\frac{{\text{sgn}(\bto(x))}f(x)}{\sqrt{|\phi'(x)|}}\right|_{x=\phi_n^{-1}(2t)}\right),
\end{split}
\ee
where $L_n^2(\mathbb R)$ is the direct sum of $n$ copies of $L^2(\mathbb R)$, $L^2_n(\mathbb R) =\oplus_{j=1}^{n} L^2(\mathbb R)$. Here we set $\Vert \check\bff\Vert^2=\Vert\check f_1\Vert^2+\dots+\Vert\check f_n\Vert^2$, where $\check\bff=(\check f_1,\dots,\check f_n)\in L_n^2(\mathbb R)$ and $\Vert\check f_m\Vert$ is the conventional $L^2(I_m)$ norm. Also, in \eqref{tr-three}, $\phi_k^{-1}$ is the inverse of $\phi(x)$ on the $k$-th interval $(\a_k,\bt_{k})$.
Changing variables in the definition of $\CH$ gives
\be\label{change-1}
\begin{split}
(T\CH T^{-1}\check \bff)_m(s)
&=\frac{\text{sgn}(\bto(z_m))}{\pi}\sqrt{\frac2{|\phi'(z_m)|}}\sum_{k=1}^n \int_{\mathbb R}\frac{\text{sgn}(\bto(x_k))\check f_k(t)}{\sqrt{|\phi'(x_k)|/2}\,(x_k-z_m)}dt\\
&=\frac{2\text{sgn}(\bto(z_m))}{\pi}\sum_{k=1}^n \int_{\mathbb R}\frac{\text{sgn}(\bto(x_k))\check f_k(t)}{\sqrt{|\phi'(x_k)||\phi'(z_m)|}\,(x_k-z_m)}dt,\\
x_k:&=\phi_k^{-1}(2t),\,z_m:=\phi_m^{-1}(2s).
\end{split}
\ee
%Here $\phi_k^{-1}$ is the inverse of $\phi$, whose domain is restricted to the $k$-th subinterval $(b_{2k-1},b_{2k})$. On each such subinterval, $\phi$ is monotone with range $\mathbb R$  (cf. \cite{kbt18}). 
Combining \eqref{change-1},  \eqref{der-2}, \eqref{cos_sinh_1}, and \eqref{bas_poly11} we find
\be\label{xmz}
\frac{\text{sgn}(\bto(x_k)\bto(z_m))}{\sqrt{|\phi'(x_k)||\phi'(z_m)|}\,(x_k-z_m)}=\frac{-1}{2\sinh(t-s)}
\sum_{j=1}^n \frac{\rho_j P_j(x_k)P_j(z_m)}{\sqrt{Q(x_k)Q(z_m)}}.
\ee
Define the matrix function
\be\label{one-matr}
\begin{split}
\CM:&=\{M_{jk}(t)\},\ M_{jk}(t):=P_j(x_k)\sqrt{\frac{\rho_j}
{Q(x_k)}},\ x_k:=\phi_k^{-1}(2t).
\end{split}
\ee
It is shown in \cite{kbt18} that $\{M_{jk}(t)\}$ is an orthogonal matrix for all $t\in \mathbb R$. Substituting \eqref{xmz} and \eqref{one-matr} into \eqref{change-1} gives
\be\label{change-simple}
\begin{split}
(T\CH T^{-1}\check \bff)_m(s)
&=\sum_{j=1}^n  M_{jm}(s) \sum_{k=1}^n \int_{\mathbb R}\frac{M_{jk}(t) \check f_k(t)}{\pi\sinh(s-t)}dt.
\end{split}
\ee
In compact form, \eqref{change-simple} can be written as follows
\be\label{change-alt}
T\CH T^{-1}\check \bff = \CM^T K \CM\check \bff,
\ee
where $K$ is the operator of component-wise convolution with $(\pi\sinh(t))^{-1}$. 

Let $\CF:\, L_n^2(\mathbb R)\to L_n^2(\mathbb R)$ denote the map consisting of $n$ component-wise one-dimensional Fourier transforms. Using the integral 2.5.46.2 in \cite{pbm1}:
\be\label{sh-int}
\int_0^\infty \frac{\sin(\la x)}{\sinh x}dx=\frac{\pi}{2}\tanh(\pi\la/2),
\ee
we get 
\be\label{K-four}
K=\CF^{-1}(i\tanh(\pi\la/2)\Id)\CF,
\ee
{where $\la$ is the spectral (Fourier) variable. }
Therefore, \eqref{change-simple} gives
\be\label{H-final}
\CH f = (\CF\CM T)^{-1} (i\tanh(\pi\la/2)\Id) (\CF\CM T)f.
\ee
Recalling that the operators $\CF$, $\CM$, and $T$ are all isometries, we immediately obtain $n$ conditions for the right-hand side of \eqref{hilb-eq} to be in the range of the multi-interval finite Hilbert transform and the inversion formula.
\begin{theorem} \label{theo-unif}
A function $g\in L^2(I)$ is in the range of the multi-interval finite Hilbert transform {$\CH$ given by \eqref{hilb-orig}} if and only if
\be\label{mifinite Hilbert transform-range}
\frac1{\la} (\CF\CM T g)_m\in L_{\rm{loc}}^2,\ m=1,2,\dots n,
\ee
in a neighborhood of $\la=0$. If $g\in L^2(I)$ is in the range of {$\CH$, then the inversion formula is} given by
\be\label{FHT-inv}
\CH^{-1}g=(\CF\CM T)^{-1} \left(\frac1{i\tanh(\la\pi/2)}\Id\right) (\CF\CM T)g.
\ee
\end{theorem}

\end{document}